\documentclass[11pt]{article}

\usepackage{mathtools}
\usepackage{mathrsfs}
\usepackage{bbm}
\usepackage{graphicx}
\usepackage{amsmath}
\usepackage{amssymb}
\usepackage{amsfonts}
\usepackage{epsfig}
\usepackage{graphicx,float}
\usepackage{color}
\usepackage{bbm}
\usepackage[table]{xcolor}
\usepackage[margin=1in]{geometry}
\usepackage{enumitem}
\usepackage{algorithm}
\usepackage{algorithmic}
\usepackage{enumerate}

\usepackage{fullpage,amsmath,amssymb,color}
\usepackage{float}
\usepackage{algorithm}
\usepackage{enumerate}
\usepackage{algorithmic}
\usepackage{graphicx}
\usepackage{epstopdf}
\usepackage{subfigure}
\usepackage{graphicx}   
\usepackage{multirow}
\usepackage{amsmath}
\usepackage{amsfonts}
\usepackage{latexsym}
\usepackage{verbatim}
\usepackage{amsmath}    
\usepackage{amsmath, amssymb, algorithm, algorithmic}
\usepackage[table]{xcolor}

\usepackage{multirow}
\usepackage [pagewise,displaymath, mathlines]{lineno}
\usepackage{amssymb}
\usepackage{epstopdf}
\usepackage{amsmath}
\usepackage{epsf}
\usepackage{graphicx}
\usepackage{mathdots}
\usepackage{setspace}

\usepackage{endnotes}
\let\footnote=\endnote

\usepackage{natbib}

\newcommand{\beq}{\begin{equation}}
\newcommand{\eeq}{\end{equation}}
\newcommand{\beqa}{\begin{eqnarray}}
\newcommand{\eeqa}{\end{eqnarray}}
\newcommand{\beqas}{\begin{eqnarray*}}
\newcommand{\eeqas}{\end{eqnarray*}}
\newcommand{\ei}{\end{itemize}}
\newcommand{\ba}{\begin{array}}
\newcommand{\ea}{\end{array}}
\newcommand{\bn}{\begin{eqnarray}}
\newcommand{\en}{\end{eqnarray}}
\newcommand{\bns}{\begin{eqnarray*}}
\newcommand{\ens}{\end{eqnarray*}}
\newcommand{\E}{{\mathbb{E}}}
\newcommand{\bbr}{\Bbb{R}}
\newcommand{\bbe}{\Bbb{E}}

\newcommand{\nn}{\nonumber}
\def\eqnok#1{(\ref{#1})}
\newcommand{\argmin}{\operatornamewithlimits{argmin}}

\def\vgap{\vspace*{.1in}}

\newtheorem{lemma}{Lemma}[section]
\newtheorem{corollary}{Corollary}[section]
\newtheorem{theorem}{Theorem}[section]
\newtheorem{proposition}{Proposition}[section]

\begin{document}

\title{Stochastic Search for a Parametric Cost Function Approximation: Energy storage with rolling forecasts}
\author{{Saeed Ghadimi\thanks{Corresponding author}\thanks{{ Department of Management Sciences, University of Waterloo, Waterloo, ON, N2L 3G1, Canada (email: sghadimi@uwaterloo.ca)}}, and
Warren B. Powell\thanks{{{}Department of Operations Research and Financial Engineering, Princeton University, Princeton, NJ, 08544, USA (email: powell@princeton.edu)}}{}}}

\maketitle

\begin{abstract}
Rolling forecasts have been almost overlooked in the renewable energy storage literature. In this paper, we provide a new approach for handling uncertainty not just in the accuracy of a forecast, but in the evolution of forecasts over time. Our approach shifts the focus from modeling the uncertainty in a lookahead model to accurate simulations in a stochastic base model. We develop a robust policy for making energy storage decisions by creating a parametrically modified lookahead model, where the parameters are tuned in the stochastic base model. Since computing unbiased stochastic gradients with respect to the parameters require restrictive assumptions, we propose a simulation-based stochastic approximation algorithm based on numerical derivatives to optimize these parameters. While numerical derivatives, calculated based on the noisy function evaluations, provide biased gradient estimates, an online variance reduction technique built in the framework of our proposed algorithm, will enable us to control the accumulated bias errors and establish the finite-time rate of convergence of the algorithm. Our numerical experiments show the performance of this algorithm in finding policies outperforming the deterministic benchmark policy.
\end{abstract}


\section{Introduction}

Over the past $20$ years, wind and solar have become important sources of clean energy and are becoming cost competitive with fossil fuels, but at a price of dealing with variability and uncertainty. For example, the total output of all the wind farms for PJM can drop from $5000$ MW to $1000$ MW in an hour or two. Overestimating the available supply from these sources may result in paying high prices to buy energy from other sources or not satisfying the demand. On the other hand, underestimation can result in missing these clean sources of energy.

In this paper, we focus on handling the uncertainty of energy from wind, where forecasts are notoriously inaccurate. There are many papers that handle the uncertainty in forecasts by solving stochastic models, but this prior work has ignored the presence of {\it rolling forecasts} which are updated every $10$ minutes. Classical methods based on Bellman's equation are not able to handle this, because it means that the forecast has to be included in the state variable, which dramatically increases the complexity of the value function. 

Standard solution strategies tend to fix a forecast and hold it constant. A deterministic lookahead will assume the forecast is perfect, while classical stochastic models will capture the error in the forecast, but ignoring the ability to continuously update the forecast within the lookahead model. Fixing the forecast in the lookahead model means that it is being treated as a latent variable. Not only does this eliminate any ability to claim optimality, it can introduce significant errors, since it ignores the ability to delay making a decision now to exploit more accurate forecasts later. 

We introduce a new approach that uses a parameterized {\it deterministic} lookahead, but where the parameters are tuned in a simulator that fully captures the rolling forecasts. Thus, we shift the emphasis from creating and solving a more realistic lookahead model, to tuning an approximate lookahead model in a more accurate simulator. The approach is simple, practical and produces results that are significantly better than a classical deterministic lookahead using rolling forecasts.

There is a body of literature that focuses on developing energy storage policies. Here, we just present a list of them based on the three widely used policies: exact or approximate value functions \cite{JiaPow15,SioMad14,XiaSio16,Xi2014ASD}, policy function approximations (such as ``buy low, sell high'' policy) \cite{ChaArGr18,TriRic19}, and model predictive control \cite{AlmaHis15,KiaLotf18,KuWeElDr18,ZafPot18}. To the best of our knowledge, in all existing energy storage models, forecasting is either ignored, or the set of forecasts over the entire horizon are fixed (see e.g., \cite{Zhang18,Dicorato12}). This assumption undermines real world problems where sources of uncertainty, in particular from renewable sources, are changing every few minutes. This highlights a broader challenge in sequential stochastic decision making problems when a given forecast for a source of uncertainty is updated frequently over the time horizon.

Several approaches have been proposed to generate forecasts of different sources of uncertainty in energy systems (which then assumed to be fixed over the horizon). For example, time series models~\cite{Benth07,Taylor05} and neural networks~\cite{ABHISHEK12,KhDaMa96}, have been used to forecast the weather (temperature) which can affect both demand and supply. There are also different approaches in forecasting renewable energy like solar radiation~\cite{Akarslan14,ARBIZ17} and wind speed~\cite{LIU2010,Justin12}. More recently, a vector autoregression model is also proposed in \cite{LIU2018} for forecasting temperature, wind speed, and solar radiation.

Our approach formalizes the idea that has been widely used in industry that an effective way to solve complex stochastic optimization problems is to shift the modeling of uncertainty from a lookahead approximation to the stochastic base model which is captured by a simulator that includes the updating of rolling forecasts, as well as capturing any other dynamics relevant to the problem. We have first presented this idea more conceptually in \cite{powelghad2022} for sequential decision making problems under uncertainity with updating forecasts. This approach, which we call parametric cost function approximations (CFAs), requires that we a) design a parameterized deterministic optimization problem and b) tune the parameters in a simulator. While the idea of parameterized policies is well known in the form of linear decision rules (also called affine policies), step functions such as order-up-to rules for inventory problems,  or even neural networks, our idea of parameterizing an optimization model is new to the stochastic optimization community. We do not minimize the challenges of the two aforementioned steps, but they are done offline, and represent the research required to design a policy that is both robust yet no more difficult to compute than basic deterministic lookaheads. In this paper, we mainly focus on applying this idea to an energy storage problem under the presence of rolling forecasts and discuss its associated computational challenges.

We make three main contributions in this paper. First, we apply the idea of using parametric CFA to handle uncertainty in the context of an energy storage problem with rolling forecasts. In contrast with a basic parametric model, our parameterized optimization model performs critical scaling functions and makes it possible to handle high-dimensional decisions. 
Second, we present a new simulation-based stochastic approximation (SA) algorithm, based on the Gaussian random smoothing technique, to optimize (tune) the parameters in the parametric CFA model while using only two function evaluations at each iteration. Our proposed algorithm is equipped with an online variance reduction technique which makes it more robust than the vanilla stochastic gradient method using numerical derivatives. Furthermore, we establish the finite-time convergence of this algorithm and show that its sample complexity is in the same order of the one presented in \cite{NesSpo17} with slightly better dependence on the problem parameters, when applied to nonsmooth nonconvex problems.
Finally, we propose several policies for parameterization of the CFA model and show that, when optimized with our proposed algorithm,  they can outperform a deterministic benchmark policy using vanilla point forecasts for an energy storage problem.

The rest of this paper is organized as follows. We discuss the issue of rolling forecasts and its importance in sequential decision making under uncertainty in Section~\ref{rolling_fcst}. We then present our energy storage model in Section~\ref{sec_energy_storage}. We discuss solution strategies in Section~\ref{Policies} and present our parametric CFA approach. We also propose a stochastic policy search algorithm to optimize the parameters within the parametric CFA model in Section~\ref{sec_algorithm} and establish its finite-time rate of convergence. We further show the performance of this algorithm in optimizing the aforementioned policies for an energy storage problem in Section \ref{sec_experiments} and  conclude the paper with some remarks in Section~\ref{sec_conclusion}.

\section{Rolling Forecasts}\label{rolling_fcst}
The problem of planning in the presence of rolling forecasts, which exhibit potentially high errors, is difficult and has been largely overlooked in the energy storage literature. A simple fact for the rolling forecasts is having accumulated noise as we predict far more in the future. For example, at time $t$, denoting the forecast of energy available from wind for time $t'$ by $\{f^E_{t,t'}\}_{t' \ge t}$ and assuming that $\{f^E_{0,t'}\}_{t'=0,\ldots,\min(H,T)}$ is given, one can generate the forecasts as
\begin{align}
f^E_{t+1,t'} &= f^E_{t,t'}+\epsilon_{t+1,t'} \qquad t=0,\ldots,T-1,\label{wind_forecast}\\
&t'=t+1,\ldots,\min(t+H,T),\nn
\end{align}
where $T$ is the problem horizon, $H$ is the size of the lookahead, $\epsilon_{t+1,t'}\sim {\cal N}(0,\sigma^2_\epsilon)$, and $\sigma_\epsilon$ depends on $\rho_E f^E_{t,t'}$ for some constant $\rho_E$. This model is usually known as the``martingale model of forecast evolution'' (see e.g., \cite{GravDaQu86,HeathJackson94,Sapra2004TheME}).

The standard approach to handling forecasts is to fix them over the planning horizon (ignoring the reality that they will actually be changing) and optimize over a deterministic future. However, an optimal policy would require modeling the evolution of forecasts over time, something that we have never seen done in a lookahead model. An alternative is to fix the forecast (say, at time $t$) over a horizon $t' \in \{t,t+1, ...,t+H\}$ and solve a stochastic dynamic program. With this strategy, the forecast becomes a {\it latent variable} in the lookahead model.  These are computationally difficult, and it would be hard doing this, for example, every $5-10$ minutes as might be required for an energy storage problem.

The failure to capture rolling forecasts represents a more significant modeling error than has been recognized in the research literature. Fixing the forecast as a latent variable ignores our ability to wait to make decisions at a later time with a more accurate forecast. Properly modeling rolling forecasts and their associated errors represents a surprisingly complex challenge in a lookahead model. If we have a rolling forecast extending $24$ hours into the future, including the forecast into the state variable introduces a $24$-dimensional component of the state variable into the model, \emph{without any particular structure that we can exploit}.

\emph{Indeed, there is an important tradeoff: including a dynamically varying forecast in the state variable produces a more complex, higher dimensional state variable, but one which does not have to be re-optimized when the forecast changes. By contrast, treating the forecast as a latent variable, as it has been done in the classical dynamic programming models using Bellman's equation, simplifies the model, but requires that the model be re-optimized when it changes.}

For this reason, we are going to adopt a completely different approach.  Rather than developing a more accurate lookahead model, we are going to use a parameterized, deterministic lookahead model, where the parameters are tuned in the simulator which captures the updating of rolling forecasts. While the parameterization needs to be carefully designed, this strategy shifts the focus from solving a complex lookahead model to using a realistic simulator, where it is much easier to handle complex dynamics. We discuss in more detail the base model and lookahead models in Section~\ref{Policies}.

\section{Energy Storage Model}\label{sec_energy_storage}
In this section, we describe an energy storage model involving rolling forecasts of wind. Assume that a smart grid manager must satisfy a recurring power demand with a stochastic supply of renewable energy, unlimited supply of energy from the main power grid at a stochastic price, and access to local rechargeable storage devices. At the beginning of each period, the manager must combine energy from different sources to satisfy the demand.

We now formally present our energy storage model by introducing the five key elements of sequential decision making under uncertainty~\cite{Powell2019,SO_Powell}, namely, state variables, decision variables, exogenous information variables, the transition function, and the objective function.

\vgap

{\bf The state variables}\\
The state variable at time $t$, $S_t$, includes the following.\\
$R_t$: The level of energy in storage satisfying $R_t\in [0, R^{\max}]$, where $R^{\max}>0$ represents the storage capacity.\\
$\{f^E_{t,t'}\}_{t' \ge t}$: The forecast of energy from wind at time $t'$ made at time $t$, where the current energy $E_t=f^E_{t,t}$.\\
$P^g_{[t]}$:  The forward curve of spot prices of electricity from the grid with the notation of $[t]=\{t'\}_{t' \ge t}$.\\
$P^m_{[t]}$: The market price of electricity.\\
$D_{[t]}$: The load curve.\\
Hence the state of the system can be represented by the vector $S_t = (R_t, f^E_{t,t'}, P^m_{[t]}, P^g_{[t]}, D_{[t]}) \ \ \forall t' \ge t$.
\vgap

{\bf The decision variables}\\
At time $t$, several decision variables should be made to satisfy the load and replenishing the storage device for the future.\\
$x_t^{wd}$: The available energy from the wind used to satisfy the load.\\
$x_t^{rd}$: The allocated energy from the storage used to satisfy the load.\\
$x_t^{gd}$: The purchased energy from the grid used to satisfy the load.\\
$x_t^{wr}$: The available energy from the wind transferred to storage.\\
$x_t^{gr}$: The purchased energy from the grid used to store.\\
$x_t^{rg}$: The stored energy to be sold to the grid.\\
Hence, the manager's decision variables at time $t$ are defined as the vector $x_t$ given by
\[
x_t= (x^{wd}_t, x^{rd}_t, x^{gd}, x^{wr}_t, x^{gr}_t, x^{rg}_t)^\top \geq 0,
\]
which should satisfy the following constraints:
\begin{eqnarray}
x^{wd}_t  + \beta^dx^{rd}_t + x^{gd}_t & \leq & D_t, \nn \\
x^{rd}_t + x^{rg}_t &\leq& R_t, \nn\\
x^{wr}_t + x^{wd}_t &\leq& E_t,\nn\\
\beta^c(x^{wr}_t + x^{gr}_t) - x^{rd}_t - x^{rg}_t &\leq& R^{\max} - R_t, \nn\\
x^{wr}_t + x^{gr}_t &\leq&  \gamma^c, \nn\\
x^{rd}_t + x^{rg}_t &\leq& \gamma^d,\label{eq:7}
\end{eqnarray}
where $\beta^c, \beta^d \in (0, 1)$ are the charge and discharge efficiencies, $\gamma^c$ and $\gamma^d$ are the maximum amount of energy that can be charged or discharged from the storage device. Figure~\ref{model1} summarizes the model.
\begin{figure}[h]
    \centering
    \includegraphics[width=3in,height=1.25in]{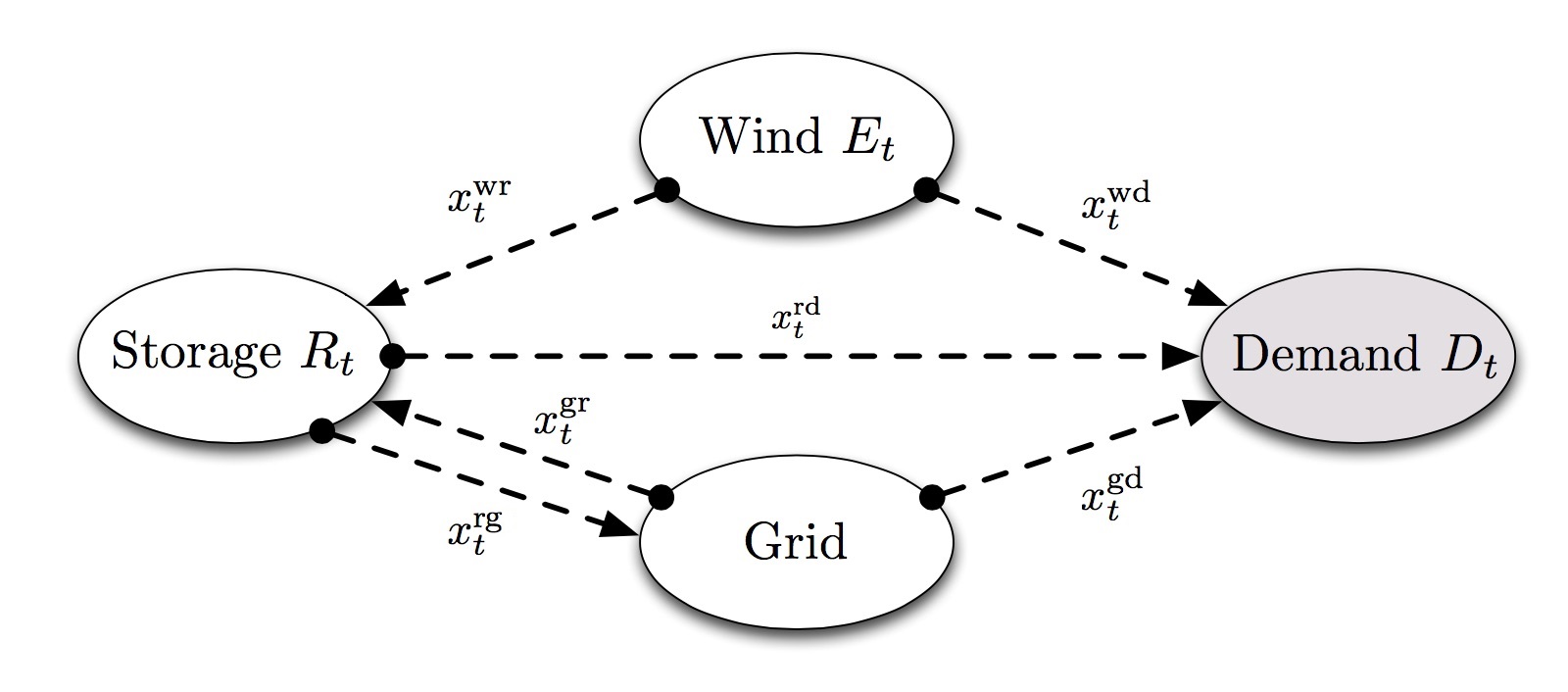}
    \caption{Energy storage model}
    \label{model1}
\end{figure}

{\bf The exogenous information}\\
This describes the information that first becomes known at time $t$. For our energy storage model, we assume that the spot price of electricity from the grid, the market price of electricity, and the load are deterministic. However, for sake of general modeling, we include them in the exogenous information. It also includes the forecasts (or the change in forecasts) of the wind. Hence, we have $W_t=(P^m_{[t]}, P^g_{[t]}, D_{[t]},f^E_{t,t'})$. It should be pointed out that the exogenous information for the next period may depend on the current decision and/or state of the system.

{\bf The transition function}\\
The transition function, $S^M(\cdot)$ also explicitly describes the relationship between the state of the model at time $t$ and $t+1$ such that $S_{t+1} = S^M(S_t, x_t, W_{t+1})$. More specifically, the relationship of storage levels between periods is defined as:
	\begin{equation}\label{eqn:trans}
		R_{t+1} = R_{t} - x^{rd}_t + \beta^c (x^{wr}_t + x^{gr}_t) - x^{rg}_t.
	\end{equation}
The forecast for the wind is also updated according to \eqref{wind_forecast}. The spot price of electricity from the grid, the market price of electricity, and the load are assumed to be fixed over the horizon and do not change once they are given at $t=1$.

{\bf The objective function}\\
To evaluate the effectiveness of a policy or sequence of decisions, we need an objective function representing the expected sum of the costs $C_t(S_t,x_t)$ in each time period $t$ over a finite horizon. Denoting the penalty of not satisfying the demand by $C^P$, for a given state $S_t$ and decision $x_t$, the cost realized at $t$ is given by
\begin{align}\label{contr_func}
	C_t(S_t, x_t) = &C^P D_t - (C^P+P^m_t) (x_t^{wd} + \beta^d x_t^{rd} + x_t^{gd}) - P^g_t (\beta^d x^{rg}_t - x^{gr}_t - x^{gd}_t).
\end{align}
Therefore, we seek to find the policy that solves
\begin{equation} \label{objective1}
	\min_{\pi \in \Pi} \: \: \mathbb{E} \bigg[ \sum^T_{t=0} C_t(S_t, X^{\pi}_t(S_t)) \: \bigg| \: S_0  \bigg],
\end{equation}
where the initial state $S_0$ is assumed to be known. If the contribution function, transition function and constraints are linear, a deterministic lookahead policy can be constructed as a linear program if point forecasts of exogenous information are provided. Equation \eqref{objective1}, along with the transition function and the exogenous information process, is called the {\it base model} which can be used to model virtually {\it any} sequential stochastic decision making problem, with possibly minor twists in the objective function for specific classes of the problem such as risk measures.

\section{Designing Policies}\label{Policies}
In this section, we first review the existing solution strategies to solve the base model and then present our new approach, namely, the parametric cost function approximation.
\subsection{The four classes of policies}
There are two general strategies for designing policies to solve the base model. The first is to use policy search, where we have to tune the parameters of a policy so that it works well over time.  The second is to build a policy that makes the best decision now that minimizes costs now and into the future (we call these lookahead policies).

The policy search class can be divided into two classes: policy function approximations (PFAs), where the policy is an analytical function that maps states to actions (such as a linear model or a neural network); and cost function approximations (CFAs) which consist of an optimization problem that has been parameterized so that it produces good solutions over time.

The lookahead class can also be divided into two classes. The class of value function approximation (VFA), is the familiar approach based on Bellman's equation where we might compute (more often we approximate) the value of being in a downstream state produced by a decision now. The class of direct lookahead approximation (DLA), is based on direct lookaheads where we optimize over some planning horizon.  The challenge with DLAs is how to handle uncertainty as we optimize over the horizon.  Most practical tools such as Google maps use a deterministic approximation.

Building uncertainty explicitly into the lookahead model is challenging.  The ultimate stochastic lookahead would require solving 
\begin{align*}
X^*_t(S_t) = &\argmin_{x_t \in {\cal X}_t} \left(C_t(S_t,x_t) + \E \Bigg\{ \min_{\pi\in\Pi}\E \Bigg\{\sum_{t'=t+1}^T C_{t'}(S_{t'},X^\pi_{t'}(S_{t'})) \Big| S_{t+1}\Bigg\} \Big| S_t,x_t \Bigg\}\right). 
\end{align*}
In special cases, the lookahead portion of the above equation can be computed exactly using Bellman's equation:
\begin{align*}
V_t(S_t)=\min_{x_t} \Big(C_t(S_t,x_t)+ \E \Big\{
V_{t+1}(S_{t+1}|S_t,x_t)\Big\}\Big),
\end{align*}
where $V_{t+1}(\cdot)$ denotes the value of the downstream impact of a decision $x_t$ made in state $S_t$. More often, we have to replace this value function with an approximation, but this only works when we can exploit structure such as convexity, linearity or monotonicity. Often, we have to directly approximate the lookahead by creating a lookahead model opening the door to a variety of approximation strategies, including the use of deterministic lookaheads, approximating the state variable and exogenous information process (this is where we can ignore the presence of rolling forecasts), along with the use of restricted policies. However, the best approach depends on the problem setting (see \cite{SO_Powell2} for more details).

\subsection{The Parametric Cost Function Approximation}\label{sec_CFA}
In this subsection, we propose using a hybrid policy of combining deterministic lookaheads with parameterically modified CFAs which can efficiently handle the issue of rolling forecasts. Consider a deterministic lookahead policy given by
\[
X^{\text{D-LA}}_t(S_t) = \argmin_{x_t, \tilde{x}_t} C_t(S_t, x_t)+ \sum_{t'=t+1}^{\min(t+H,T)}C_{t'}(S_{t'},\tilde x_{t,t'}),
\]
\begin{eqnarray}
\text{s.t.} \quad  \eqnok{eq:7}, \eqnok{eqn:trans} \:\: x_t, \tilde x_t \geq 0, \quad &\text{and}&\nn\\
\tilde x^{wd}_{t,t'}  + \beta^d \tilde x^{rd}_{t,t'} + \tilde x^{gd}_{t,t'} & \leq & D_{t'}, \nn \\
\tilde x^{rd}_{t,t'} + \tilde x^{rg}_{t,t'} &\leq& R_{t'}, \nn\\
\tilde x^{wr}_{t,t'} + \tilde x^{wd}_{t,t'} &\leq& f^E_{t,t'},\nn\\
\beta^c(\tilde x^{wr}_{t,t'} + \tilde x^{gr}_{t,t'}) - \tilde x^{rd}_{t,t'} - \tilde x^{rg}_{t,t'}+ R_{t'} &\leq& R^{\max}, \nn\\
\tilde x^{wr}_{t,t'} + \tilde x^{gr}_{t,t'} &\leq&  \gamma^c, \nn\\
\tilde x^{rd}_{t,t'} + \tilde x^{rg}_{t,t'} &\leq& \gamma^d,\nn \\
R_{t'} - \tilde x^{rd}_{t,t'} + \beta^c (\tilde x^{wr}_{t,t'} + \tilde x^{gr}_{t,t'}) - \tilde x^{rg}_{t,t'} &=&R_{t'+1},\label{DL}
\end{eqnarray}
where $\tilde x_t = (\tilde x_{t,t'})_{t'=t+1}^{\min(t+H,T)}$. When we solve the above model, we keep $x_t$ to compute the portion of the cost function at time $t$, and discard all $\tilde x_{t,t'}$ and repeat this process as we move forward over the problem horizon.

In the parametric CFA approach, we parameterize the lookahead model in \eqref{DL} in which the parametric terms can be added to the cost function and/or constraints. In this paper, we focus on parameterizing constraints including noisy forecasts. Hence, our hybrid policy $X^{\pi}_t(S_t|\theta)$ is defined as the solution to the linear programming model \eqref{DL} in which the wind energy constraint is updated as
\beq
\tilde x^{wr}_{t,t'} + \tilde x^{wd}_{t,t'} \leq b_{t'}(f^E_{t,t'},\theta),\label{param_const}
\eeq
where $b_{t'}$ is a real valued function and $\theta$ is the set of constraint parameters.

We then need to optimize the values of parameters $\theta$, for a given policy $\pi$, by solving
\begin{align}
	\min_{\theta} \: & \Bigg\{F^\pi(\theta):= \mathbb{E}_\omega\left[\bar{F}^\pi(\theta, \omega) \right]= \mathbb{E}\left[ \sum_{t=0}^T C_t \bigg(S_{t}(\omega),X^\pi_{t}(S_{t}(\omega)|\theta) \bigg) \bigg|S_0 \right]\Bigg\}. \label{cum_reward}
\end{align}
It should be pointed out that the more general optimization problem associated with the parametric CFA approach is to optimize over the structure of policies and their parameterization simultaneously. However, our focus in this paper is to solve problem \eqref{cum_reward} to optimize the parameters for a given structure of a policy. The tuned parameters capture the proper dynamics of forecasts, unlike an optimal solution to a stochastic lookahead that uses a fixed forecast. However, tuning is not easy since the above problem is usually nonconvex and we will discuss an approximation algorithm in Section~\ref{sec_algorithm} to solve it. We refer interested readers to the companion paper \cite{powelghad2022} in which we described the idea of the parametric CFA approach in more detail and for general decision making problems under uncertainty.

An important step in the parametric CFA approach is to consider meaningful parameterized policies in the model. This step is truly domain dependent and can be significantly different from one problem to another one. Indeed, this step is the art of modeling that draws on a statistical model or the knowledge and insights of the domain experts. For the energy storage problem in this paper, we assume that uncertainties only exists in the wind forecasts. Therefore, we propose the following:
\begin{itemize}

\item {\bf Constant parameterization ($\pi=const$)} - This parameterization uses a single scalar to modify the forecast of energy from wind for the entire horizon such that $b_{t'}$ in \eqref{param_const} is set to $b_{t'}(f^E_{t,t'},\theta) =\theta \cdot f^E_{t,t'}$.

	\item {\bf Lookup table parameterization ($\pi=lkup$)} - Overestimating or underestimating forecasts of energy from wind influences how aggressively a policy will store energy. We can modify the forecast for each period of the lookahead model with a unique parameter $\theta_{\tau}$. This parameterization is a lookup table representation because there is a different $\theta_\tau$ for each lookahead period, $\tau = 0, 1, 2, ...$ This implies that $b_{t'}(f^E_{t,t'},\theta) =\theta_{t' - t} \cdot f^E_{t,t'}$, where $t' \in [t+1,\min(t+H,T)]$ and $\tau = t' -t$. If $\theta_{\tau} < 1$ the policy will be more conservative and decrease the risk of running out of energy. Conversely, if $\theta_{\tau} > 1$ the policy will be more aggressive and less adamant about maintaining large energy reserves. This is a time-independent (or stationary) parameterization since the modification of the forecast at each time period depends on how far in the future forecasts are provided.

	\item {\bf Exponential decay parameterization ($\pi=exp$)} Instead of calculating a set of parameters for every period within the lookahead model, we can make our parameterization a function of time and a few parameters. Intuitively, we can assume the forecasts become worse when we are far in the future. Hence, it might be good to try some decaying functions of parameters to decrease the impact of errors in forecasts for the far future. To do this, we suggest using the following exponential function of two variables which also limits the search space of parameters into a two dimensional plane i.e.,  $b_{t'}(f^E_{t,t'},\theta)= f^E_{t,t'} \cdot \theta_1 \cdot e^{\theta_2 \cdot (t'-t)}$.
	
\end{itemize}
Similar parameterization schemes can be also proposed for the RHS of other constraints in the lookeahead model, if they include noisy forecasts. The combination of these parameterizations can be then used in the parametric CFA model, but tuning the higher dimensional parameter vector becomes harder.

\section{The Stochastic Search Algorithm}\label{sec_algorithm}
Our goal in this section is to solve problem \eqnok{cum_reward} under specific assumptions on $F^\pi(\theta)$. Even for a simple parameterization, this function is possibly nonconvex and nonsmooth which makes the optimization problem hard to solve. On the other hand, computing unbiased (sub)gradient estimates of the objective function w.r.t the parameters may be prohibitive or impossible. We first present a result on computing unbiased stochastic (sub)gradient of $F^\pi(\theta)$ under certain conditions. We then discuss the setting in which we cannot compute these stochastic (sub)gradients and we only have access to noisy evaluations of $F^\pi(\theta)$. We present a simulation-based optimization algorithm based on a randomized Gaussian smoothing technique and establish its finite-time rate of convergence to a stationary point of problem \eqnok{cum_reward} when $F^\pi(\theta)$ is possibly nonsmooth and nonconvex.

Stochastic approximation algorithms require computing stochastic (sub)gradients of the objective function iteratively. Due to the special structure of $F^\pi(\theta)$, its (sub)gradient can be computed recursively under certain conditions as shown in the next result.

\begin{proposition}\label{prop_grad}
	Assume $\bar{F}^\pi(\cdot, \omega)$ is convex/concave for every $\omega \in \Omega$, and $F^\pi(\cdot)$ is finite valued in the neighborhood of $\theta$. If distribution of $\omega$ is independent of $\theta$, we have
	\begin{equation*}
		\nabla_{\theta}F^\pi(\theta) = \mathbb{E}[\nabla_{\theta}\bar{F}^\pi(\theta, \omega)],
	\end{equation*}
	where
	\begin{equation}\label{eq:stoch_gradient}
		\nabla_{\theta}\bar{F}^\pi(\theta) = \bigg(\frac{\partial C_0}{\partial X^\pi_0} \cdot \frac{\partial X^\pi_0}{\partial \theta} \bigg) + \sum^T_{t=1} \bigg[ \bigg( \frac{\partial C_{t}}{\partial S_{t}} \cdot \frac{\partial S_{t}}{\partial \theta} \bigg) + \bigg( \frac{\partial C_{t}}{\partial X^\pi_{t}} \cdot \bigg( \frac{\partial X^\pi_{t}}{\partial S_{t}} \cdot \frac{\partial S_{t}}{\partial \theta} + \frac{\partial X^\pi_{t}}{\partial \theta} \bigg) \bigg) \bigg], \\
        \end{equation}
        \begin{equation*}
\text{and} \qquad  	\frac{\partial S_{t}}{\partial \theta} = \frac{\partial S_{t}}{\partial S_{t-1}} \cdot \frac{\partial S_{t-1}}{\partial \theta} + \frac{\partial S_{t}}{\partial X^\pi_{t-1}} \cdot \bigg[ \frac{\partial X^\pi_{t-1}}{\partial S_{t-1}} \cdot \frac{\partial S_{t-1}}{\partial \theta} + \frac{\partial X^\pi_{t-1}}{\partial \theta} \bigg],
        \end{equation*}
in which the $\omega$ is dropped for simplicity.
\end{proposition}

\begin{proof}
If $\bar{F}^\pi(\cdot, \omega)$ is convex or concave for every $\omega \in \Omega$, and $F^\pi(\cdot)$ is finite valued in the neighborhood of $\theta$, then we have $\nabla_{\theta} \left(\mathbb{E} \left[\bar{F}^\pi(\theta,\omega)\right]\right) = \mathbb{E} \left[\nabla_{\theta} \bar{F}^\pi(\theta,\omega)\right]$ by \cite{Stra65}. Applying the chain rule, we find
	\begin{equation*}
	\begin{split}
	 \nabla_{\theta}\bar{F}^\pi(\theta) &= \frac{d}{d \theta} \bigg[ C_0(S_0, X^\pi_0) + \sum^T_{t=1} C(S_{t},X^\pi_{t}) \bigg]
                        		= \bigg(\frac{\partial C_0}{d X^\pi_0} \cdot \frac{d X^\pi_0}{\partial \theta} \bigg) + \bigg[  \sum^T_{t=1} \frac{d}{d \theta} \: C(S_{t},X^\pi_{t}) \bigg] \\
                        		&= \bigg(\frac{\partial C_0}{d X^\pi_0} \cdot \frac{d X^\pi_0}{\partial \theta} \bigg) + \sum^T_{t=1} \bigg[ \bigg( \frac{\partial C_{t}}{\partial S_{t}} \cdot \frac{\partial S_{t}}{\partial \theta} \bigg) + \bigg( \frac{\partial C_{t}}{\partial X^\pi_{t}} \cdot \frac{d X^\pi_{t}}{d \theta} \bigg) \bigg] \\
                        		&= \bigg(\frac{\partial C_0}{d X^\pi_0} \cdot \frac{d X^\pi_0}{\partial \theta} \bigg) + \sum^T_{t=1} \bigg[ \bigg( \frac{\partial C_{t}}{\partial S_{t}} \cdot \frac{\partial S_{t}}{\partial \theta} \bigg) + \bigg( \frac{\partial C_{t}}{\partial X^\pi_{t}} \cdot \bigg( \frac{\partial X^\pi_{t}}{\partial S_{t}} \cdot \frac{\partial S_{t}}{\partial \theta} + \frac{\partial X^\pi_{t}}{\partial \theta} \bigg) \bigg) \bigg], \\
                        \end{split}
                        \end{equation*}
where
\begin{equation*}
\frac{\partial S_{t}}{\partial \theta} = \frac{\partial S_{t}}{\partial S_{t-1}} \cdot \frac{\partial S_{t-1}}{\partial \theta} + \frac{\partial S_{t}}{\partial X^\pi_{t-1}} \left[ \frac{\partial X^\pi_{t-1}}{\partial S_{t-1}} \cdot \frac{\partial S_{t-1}}{\partial \theta} + \frac{\partial X^\pi_{t-1}}{\partial \theta} \right].
\end{equation*}
\end{proof}

Note that if $\bar F^\pi(\theta)$ is not differentiable, then its subgradient can be still computed using \eqnok{eq:stoch_gradient}. However, when $\bar F^\pi(\theta)$ is not convex (concave), its subgradient may not exist and the concept of generalized subgradient should be employed. If $\nabla_{\theta}\bar{F}^\pi(\theta,\omega)$ exists for every $\omega \in \Omega$, the ability to calculate its unbiased estimator allows us to use SA-type techniques such as stochastic gradient descent (SGD) to determine the optimal parameter $\theta^*$. However, this is not always the case. The function $F^\pi(\theta)$ can be generally nonsmooth and nonconvex and hence, its subgradient may not exist everywhere. Moreover, calculating \eqnok{eq:stoch_gradient} may not be easy. Therefore, we propose an alternative way to estimate gradient of $F^\pi(\theta)$.

To simplify our notation, we drop the superscript $\pi$ for the policies in definition of the objective function and it only refers to the number $\pi$ in the rest of this section. Before we proceed, we assume that the objective function is Lipschitz continuous w.r.t $\theta$ with constant $L_0>0$, for any $\omega \in \Omega$ i.e.,
\[
|\bar F(\theta_1,\omega) - \bar F(\theta_2,\omega)| \le L_0 \|\theta_1 -\theta_2\| \qquad \forall \theta_1, \theta_2,
\]
which consequently implies that $F(\theta)$ is Lipschitz continuous with constant $L _0$. This is a reasonable assumption for most of the applications as the cost (objective function) does not make sudden changes w.r.t small change of resources (policies). This property will be used to establish the convergence analysis of our proposed algorithm. Furthermore, we assume that noisy evaluations of $F(\theta)$ can be obtained through simulations and hence, we can use techniques from simulation-based optimization where even the shape of the function may not be known (see e.g., \cite{fu2015handbook} and the references therein). In the reminder of this section, we provide a zeroth-order SA algorithm and establish its finite-time convergence analysis to solve problem \eqnok{cum_reward}. 

A smooth approximation of the function $F(\theta)$ can be defined by the following convolution:
\beq \label{rand_smooth_func}
F_{\eta}(\theta) = \frac{1}{(2 \pi)^{\frac{d}{2}}} \int F(\theta+\eta v) e^{-\frac{1}{2}\|v\|^2} \,dv =\bbe_v[F(\theta+\eta v)].
\eeq
where $\eta>0$ is the smoothing parameter and $v \in \bbr^d$ is a Gaussian random vector whose mean is zero and covariance is the identity matrix. The following result in \cite{NesSpo17} provides some properties of $F_{\eta}(\cdot)$.

\begin{lemma} \label{smth_approx}
The following statements hold for any Lipschitz continuous function $F$ with constant $L_0$.
\begin{itemize}
\item [a)] The function $F_\eta$ is differentiable and its gradient is given by
\beq \label{smth_approx_grad}
\nabla F_{\eta}(\theta) = \frac{1}{(2 \pi)^{\frac{d}{2}}} \int \tfrac{F(\theta+\eta v)-F(\theta)}{\eta} v e^{-\frac{1}{2}\|v\|^2} \,dv = 
\bbe_v [\tfrac{F(\theta+\eta v)-F(\theta)}{\eta} v].
\eeq
\item [b)] The gradient of $F_\eta$ is Lipschitz continuous with constant $L_{\eta} = \frac{\sqrt{d}}{\eta}L_0$, and for any $\theta \in \bbr^d$, we have
\beqa \label{rand_smth_close}
|F_{\eta}(\theta)-F(\theta)| &\le& \eta L_0 \sqrt{d},\\\
\bbe_v [\|[F(\theta+\eta v)-F(\theta)]v\|^2] &\le& \eta^2 L_0^2(d+4)^2.\label{bnd_grad}
\eeqa
\end{itemize}

\end{lemma}

\vgap

We also need the following result about using different smoothing parameters. 
\begin{lemma}
Assume that the function $F$ is Lipschitz continuous with constant $L_0$ and $\eta_1,\eta_2>0$. Then, for any $\theta \in \bbr^d$, we have
\beq\label{lips_diff}
\|\nabla F_{\eta_2}(\theta) - \nabla F_{\eta_1}(\theta)\| \le \frac{2 L_0 d|\eta_2 -\eta_1|}{\eta_2}.
\eeq
\end{lemma}
\begin{proof}
Noting \eqref{smth_approx_grad}, we have
\begin{align*}
\|\nabla F_{\eta_2}(\theta) - \nabla F_{\eta_1}(\theta)\| &= \left\|\bbe_v \left[\left(\tfrac{\eta_1 F(\theta+\eta_2 v)-\eta_2F(\theta+\eta_1 v)+(\eta_2-\eta_1) F(\theta)}{\eta_2 \eta_1}\right) v \right]\right\|\\
&\le \bbe_v \left[\left(\tfrac{|F(\theta+\eta_2 v)-F(\theta+\eta_1 v)|}{\eta_2}+ \tfrac{|\eta_2 - \eta_1||F(\theta+\eta_1 v)- F(\theta)|}{\eta_2 \eta_1}\right) \left\|v \right\|\right]\\
&\le \frac{2 L_0 d|\eta_2 -\eta_1|}{\eta_2},
\end{align*}
where the last inequality follows from the Lipschitz continuity of $F$ and the fact that $\bbe_v[\|v\|^2]=d$.
\end{proof}

\vgap

Our method is formally described as Algorithm~\ref{CFA_Grad_F}. The gradient estimate at its $k$-th iteration, as shown in \eqref{grad_zorth}, is obtained by only sampling a random Gaussian vector $v^k$ and computing the noisy objective values at the current point $\theta^k$ and its perturbation $\theta^k+\eta_k v^k$.
While \eqref{grad_zorth} gives us an unbiased estimator for $\nabla F_{\eta_k}(\theta^k)$ as
\beq\label{grad_unbiased}
\bbe_{\omega,v}[G_{\eta_k}(\theta^k,\omega^k)] = \nabla F_{\eta_k}(\theta^k)
\eeq
due to \eqref{smth_approx_grad}, it provides a biased estimator for $\nabla F(\theta^k)$. However, by properly choosing the algorithm parameters, we can control the bias terms and ensure convergence of the algorithm. In particular, if $\alpha_k=1 \ \ \forall k$, Algorithm~\ref{CFA_Grad_F} reduces to those presented in \cite{GhaLan12,NesSpo17} with the established rate of convergence, which are based on the classical simultaneous perturbation stochastic approximation (SPSA) algorithm~\cite{spall1992multivariate}.

\begin{algorithm}[h]
   \caption{The stochastic averaging numerical gradient method (SANG)}
   \label{CFA_Grad_F}
   \begin{algorithmic}[1]
    	\STATE Input: $\theta^0 \in \bbr^d$, $\bar G^0 =0$, an iteration limit $N$, positive sequences $\{\eta_k\}_{k \ge 1}$, $\{\gamma_k\}_{k \ge 1}$, $\{\alpha_k\}_{k \ge 1} \in (0,1)$, and a probability mass function (PMF) $P_R(\cdot)$ supported on $\{1,\dots,N\}$.
	
\STATE Generate a random number $R$ according to $P_R(\cdot)$.

\vgap

{\bf For $k=1,\ldots, R$:}

{\addtolength{\leftskip}{0.2in}

{\hspace{1in}} \STATE Update policy parameters as
							\begin{align}
								\theta_y^k &= \theta^{k-1}-\beta_k \bar G^{k-1},\label{update_thetay}\\
                                \theta^k &= (1-\alpha_k)\theta^{k-1} + \alpha_k \theta_y^k.\label{update_theta}
							\end{align}

		\STATE Generate a trajectory $\omega^{k}$ where $S^k_{t+1}(\omega^k) = S^M(S^k_t(\omega^k), X_t^\pi(S^k_t(\omega^k)|\theta^k),W_{t+1}(\omega^k))$, and a random Gaussian vector $v_k$ to compute the gradient estimator as
 \beq \label{grad_zorth}
 G_{\eta_k}(\theta^k,\omega^k) = \tfrac{\bar{F}(\theta^k+\eta_k v^k,\omega^k)-\bar{F}(\theta^k,\omega^k)}{\eta_k}v^k,
 \eeq							
and set
\beq\label{grad_ave}
   \bar G^k = (1-\alpha_k) \bar G^{k-1}+\alpha_k G_{\eta_k}(\theta^k,\omega^k).
\eeq

}

{\bf End For}

    \end{algorithmic}
    \end{algorithm}

On the other hand, if the unbiased stochastic gradient of $f$ is available and used instead of $G_{\eta_k}(\theta^k,\omega^k)$, Algorithm~\ref{CFA_Grad_F} reduces to a variant of the algorithm proposed in \cite{ghadimi2020single} for nested problems. Thus, one may easily establish the convergence analysis of Algorithm~\ref{CFA_Grad_F} assuming it is applied to minimize $F_\eta (\theta)$. However, this does not specify the choice of smoothing parameter $\eta$. In addition, the smoothing parameter can be different at each iteration and hence, the convergence analysis of \cite{ghadimi2020single} is not directly applicable as the smoothing function is changed every iteration.

In the next result, we provide the main convergence analysis of our proposed algorithm.

\begin{theorem}\label{theom_convg}
Let $\{\theta_k\}$ be generated by Algorithm~\ref{CFA_Grad_F}, $\bar F (\theta)$ be Lipschitz continuous with constant $L_0$, and $F(\theta)$ be bounded below by $F^*$. If parameters are chosen such that
\beq\label{choice_param}
\beta_0 =1, \quad \eta_{k+1} \le \eta_k, \quad \beta_{k+1} \le \beta_k, \quad \beta_k \le \frac{c_2 \eta_k}{L_0\sqrt{d}}, \quad \sum_{i=k}^N \alpha_i A_i \le c_1 A_k \qquad \forall k \ge 1
\eeq
for some positive constants $c_1, c_2 >0$ and 
\beq \label{def_Ak}
A_k =  \prod_{i=1}^k (1-\alpha_i) \qquad \forall k \ge 1.
\eeq
We have
\begin{align}
\sum_{k=2}^N  \alpha_k \beta_k \bbe[\|\bar G^{k-1}\|^2] &\le
F^* - F(\theta^0)+2 \eta_1 L_0 \sqrt{d}+ \frac{L_0^2 (d+4)^2 (4+c_2)}{2}\sum_{k=1}^N \beta_k \alpha_k^2, \label{stationry_Geta}\\
\sum_{k=1}^N  \alpha_k \beta_k \bbe[\|\nabla F_{\eta_k} (\theta^k)\|^2] &\le
2(1+2 c_1 c_2^2) (F^* - F(\theta^0)+2 \eta_1 L_0 \sqrt{d})\nn\\
&+  L_0^2 (d+4)^2 \sum_{k=1}^N \beta_k \left[\frac{8c_1 (\eta_k -\eta_{k-1})^2}{\alpha_k \eta_k^2} + [2(1+c_1)+(4+c_2)(1+2 c_1 c_2^2)] \alpha_k^2\right],\label{stationry_Feta}
\end{align}
where the expectation is taken w.r.t the random vector $\omega$, and Gaussian random vector $v$.
\end{theorem}
\begin{proof}
First note that $F (\theta)$ is Lipschitz continuous with constant $L$ due to the same assumption on $\bar F (\theta)$. Hence, the gradient of $F_\eta (\theta)$ is Lipschitz continuous with constant $L_\eta$ due to Lemma~\ref{smth_approx}.b which together with the fact that
\beq\label{diff_theta}
\theta^k - \theta^{k-1} = \alpha_k (\theta_y^k - \theta^{k-1}) = -\alpha_k \beta_k \bar G^{k-1}
\eeq
due to \eqref{update_thetay} and \eqref{update_theta}, imply that
\beqa
F_{\eta_k} (\theta^{k-1}) &\ge& F_{\eta_k} (\theta^k) + \langle \nabla F_{\eta_k} (\theta^k), \theta^{k-1} - \theta^k \rangle -\frac{L_{\eta_k}}{2}\|\theta^k - \theta^{k-1}\|^2 \nn\\
&=& F_{\eta_k} (\theta^k) + \alpha_k \beta_k \langle \nabla F_{\eta_k} (\theta^k), \bar G^{k-1} \rangle -\frac{L_{\eta_k} \alpha_k^2 \beta_k^2}{2}\|\bar G^{k-1}\|^2\nn\\
&=& F_{\eta_k} (\theta^k) +\alpha_k \beta_k \langle \delta_k +  G_{\eta_k}(\theta^k,\omega^k), \bar G^{k-1} \rangle 
-\frac{L_{\eta_k} \alpha_k^2 \beta_k^2}{2}\|\bar G^{k-1}\|^2,\label{recursion0}
\eeqa
where $\delta_k := \nabla F_{\eta_k} (\theta^k)-G_{\eta_k}(\theta^k,\omega^k)$. Moreover, by \eqref{grad_ave}, we have
\begin{align*}
\|\bar G^k\|^2 &= \|\bar G^k - \bar G^{k-1}\|^2 + \|\bar G^{k-1}\|^2 +2 \langle \bar G^k - \bar G^{k-1}, \bar G^{k-1}\rangle \\
&= \alpha_k^2\|G_{\eta_k}(\theta^k,\omega^k) - \bar G^{k-1}\|^2 + \|\bar G^{k-1}\|^2 +2 \alpha_k \langle G_{\eta_k}(\theta^k,\omega^k) - \bar G^{k-1}, \bar G^{k-1}\rangle.
\end{align*}
Multiplying the above relation by $\beta_k/2$, combining it by \eqref{recursion0}, noting the fact that $\beta_k \le \beta_{k-1}$ re-arranging the terms, we obtain
\begin{align}
F_{\eta_k} (\theta^k) + \tfrac{\beta_k}{2} \|\bar G^k\|^2 &\le 
F_{\eta_k} (\theta^{k-1}) + \tfrac{\beta_{k-1}}{2} \|\bar G^{k-1}\|^2 -\alpha_k \beta_k \|\bar G^{k-1}\|^2
- \alpha_k \beta_k \langle \delta_k , \bar G^{k-1} \rangle \nn\\
&+ \tfrac{\beta_k \alpha_k^2}{2} \Big[\|G_{\eta_k}(\theta^k,\omega^k) - \bar G^{k-1}\|^2
+L_{\eta_k} \beta_k\|\bar G^{k-1}\|^2 \Big].\label{rec_f}
\end{align}
Dividing both sides of \eqref{grad_ave} by $A_k$ (defined in \eqref{def_Ak}), summing them up, and noting that $\bar G^0 =0$, we obtain
\beq \label{rec1_g}
\bar G^k = A_k \sum_{i=1}^k \frac{\alpha_i}{A_i} G_{\eta_i}(\theta^i,\omega^i),
\eeq
which together with the convexity of $\|\cdot\|^2$ and the fact that
\[
\sum_{i=1}^k \frac{\alpha_i}{A_i}=\frac{\alpha_1}{A_1}+\sum_{i=2}^k \frac{\alpha_i}{A_i} = \frac{\alpha_1}{A_1}+\sum_{i=2}^k \left(\frac{1}{A_i}-\frac{1}{A_{i-1}}\right) = \frac{1}{A_k}-1
\]
due to \eqref{def_Ak}, imply
\[
\|\bar G^k\|^2 \le (1-A_k) A_k \sum_{i=1}^k \frac{\alpha_i}{A_i} \|G_{\eta_i}(\theta^i,\omega^i)\|^2.
\]
Moreover, by Lipschitz continuity of $\bar{F}(\cdot,\omega)$, \eqref{bnd_grad}, and \eqref{grad_unbiased}, we have
\begin{align*}
&\bbe[\delta_k] = 0,\\
&\bbe[\|G_{\eta_k}(\theta^k,\omega^k)\|^2] \le L_0^2(d+4)^2,\\
&\bbe[\|\bar G^k\|^2] \le L_0^2(d+4)^2 A_k \sum_{i=1}^k \frac{\alpha_i}{A_i} \le L_0^2(d+4)^2,\\
&\bbe[\|G_{\eta_k}(\theta^k,\omega^k) - \bar G^{k-1}\|^2] \le 4 L_0^2(d+4)^2,
\end{align*}
which together with \eqref{rec_f} and the assumption that $L_{\eta_k} \beta_k \le c_2$, imply that 
\[
2\alpha_k \beta_k \|\bar G^{k-1}\|^2 + \beta_k \|\bar G^k\|^2  \le \beta_{k-1} \|\bar G^{k-1}\|^2 + 
2 [F_{\eta_k} (\theta^{k}) - F_{\eta_k} (\theta^{k-1})] +  L_0^2(d+4)^2 (4+c_2)\beta_k \alpha_k^2.
\]
Summing up the above inequalities, and rearranging the terms, and noting the fact that $\bar G^0=0$, we obtain
\[
2\sum_{k=2}^N  \alpha_k \beta_k \bbe[\|\bar G^{k-1}\|^2] \le
2 \Delta_N+ L_0^2(d+4)^2 (4+c_2) \sum_{k=1}^N \beta_k \alpha_k^2,
\]
where $\Delta_N = F_{\eta_N}(\theta^N)-F_{\eta_1}(\theta^0)+\sum_{k=1}^{N-1} [F_{\eta_k}(\theta^k)-F_{\eta_{k+1}}(\theta^k)]$. Noting \eqnok{rand_smth_close}, the fact that $F (\theta) \le F^*$ for any $\theta \in \bbr^d$, \eqnok{rand_smooth_func}, Lipschitz continuity of $F$, and the assumption that $\eta_{k+1} \le \eta_k$ we have
\beqa
F_{\eta_N}(\theta^N)-F_{\eta_1}(\theta^0) &\le& F^* - F(\theta^0)+ (\eta_1+\eta_N) L_0 \sqrt{d}, \nn \\
F_{\eta_k}(\theta^k)-F_{\eta_{k+1}}(\theta^k) &=&  \bbe_v[F(\theta^k+\eta_k v)-F(\theta^k+\eta_{k+1} v)] \nn \\
 &\le& L_0 (\eta_k - \eta_{k+1}) \bbe_v[\|v\|] \le L_0 \sqrt{d} (\eta_k - \eta_{k+1}),\nn
\eeqa
which clearly implies that $\Delta_N \le F^* - F(\theta^0)+ 2\eta_1 L_0 \sqrt{d}$. We then immediately conclude \eqref{stationry_Geta}.

Now, observe that we need to bound $\|\nabla F_{\eta_k} (\theta^k)- \bar G^k \|^2$ as
\begin{align}
\|\nabla F_{\eta_k} (\theta^k)\|^2 &\le 2\Big[\|\nabla F_{\eta_k} (\theta^k)- \bar G^k \|^2+ \|\bar G^k \|^2 \Big]\nn \\
&\le 2 \Big[\|\nabla F_{\eta_k} (\theta^k)- \bar G^k \|^2+ (1-\alpha_k) \|\bar G^{k-1} \|^2 + \alpha_k \|G_{\eta_k}(\theta^k,\omega^k)\|^2 \Big],\label{grad_upper}
\end{align}
where the second inequality comes from the convexity of $\|\cdot\|^2$ and \eqref{grad_ave}. Noting \eqref{grad_ave} and the definition of $\delta_k$, we have
\beq \label{rec0}
\nabla F_{\eta_k} (\theta^k)- \bar G^k = (1-\alpha_k) (\nabla F_{\eta_k} (\theta^k)-\bar G^{k-1}) +\alpha_k \delta_k,
\eeq
which implies that
\begin{align}
\|\nabla F_{\eta_k} (\theta^k)- \bar G^k\|^2 &= \|(1-\alpha_k) (\nabla F_{\eta_k} (\theta^k)-\bar G^{k-1})\|^2 +\alpha_k^2 \|\delta_k\|^2 \nn \\
&+ 2 \alpha_k (1-\alpha_k) \langle \nabla F_{\eta_k} (\theta^k)-\bar G^{k-1} , \delta_k \rangle.\label{sqr1}
\end{align}
Moreover, by the convexity of $\|\cdot\|^2$ and we have
\begin{align}\label{grad_recu0}
\|(1-\alpha_k) (\nabla F_{\eta_k} (\theta^k)-\bar G^{k-1})\|^2 &=
\|(1-\alpha_k) (\nabla F_{\eta_{k-1}} (\theta^{k-1})-\bar G^{k-1})+ \alpha_k e_k\|^2\nn \\
&\le (1-\alpha_k) \|\nabla F_{\eta_{k-1}} (\theta^{k-1})-\bar G^{k-1}\|^2+ \alpha_k \|e_k\|^2,
\end{align}
where
\beq\label{def_ek}
e_k = \frac{(1 - \alpha_k)}{\alpha_k} \Big[\nabla F_{\eta_k} (\theta^k) - \nabla F_{\eta_{k-1}} (\theta^{k-1})\Big].
\eeq
Hence, by \eqref{sqr1} and \eqref{grad_recu0}, 
\begin{align}
\|\nabla F_{\eta_k} (\theta^k)- \bar G^k\|^2 &\le (1-\alpha_k) \| \nabla F_{\eta_{k-1}} (\theta^{k-1})-\bar G^{k-1}\|^2 + \alpha_k \|e_k\|^2+\alpha_k^2 \|\delta_k\|^2 \nn \\
&+ 2 \alpha_k (1-\alpha_k) \langle \nabla F_{\eta_k} (\theta^k)-\bar G^{k-1} , \delta_k \rangle.\label{sqr2}
\end{align}
Thus, similar to \eqref{rec1_g}, we can obtain
\begin{align} 
\|\nabla F_{\eta_k} (\theta^k)- \bar G^k\|^2 &\le A_k \left( \|\nabla F_{\eta_0} (\theta^0)\|^2 + \sum_{i=1}^k \left[ \frac{\alpha_i}{A_i} \|e_i\|^2 + \frac{\alpha_i^2}{A_i} \|\delta_i\|^2 \right] \right)\nn\\
&+ 2A_k \sum_{i=1}^k \frac{\alpha_i (1-\alpha_i)}{A_i} \langle \nabla F_{\eta_i} (\theta^i)-\bar G^{i-1} , \delta_i \rangle.\label{rec1}
\end{align}
In the rest of the proof, only for the sake of simplicity, we assume that $\eta_0$ and $\theta_0$ are chosen such that $\nabla F_{\eta_0} (\theta^0)=0$. Now, by \eqref{diff_theta}, Lipschitz continuity of $\bar{F}(\cdot,\omega)$ and $\nabla F_\eta$, \eqref{bnd_grad}, and \eqref{lips_diff}, we have
\begin{align*}
&\|\nabla F_{\eta_k} (\theta^k)-\nabla F_{\eta_{k-1}} (\theta^{k-1})\|^2 \nn\\
&\le 2 \Big(\|\nabla F_{\eta_k} (\theta^k)-\nabla F_{\eta_k} (\theta^{k-1})\|^2 + \|\nabla F_{\eta_k} (\theta^{k-1})-\nabla F_{\eta_{k-1}} (\theta^{k-1})\|^2 \Big)\\
&\le 2 \left[(L_{\eta_k} \alpha_k \beta_k)^2 \|\bar G^{k-1}\|^2 + \frac{4 L_0^2 d^2 (\eta_k -\eta_{k-1})^2}{\eta_k^2} \right],\\
&\bbe\left[\|\delta_k\|^2\right] 
\le 2 \Big[\|\nabla F_{\eta_k} (\theta^k)\|^2 + \bbe[\|G_{\eta_k}(\theta^{k-1},\omega^k)\|^2 \Big]
\le 2L_0^2(d+4)^2,\\
&\bbe[\langle \nabla F_{\eta_k} (\theta^k)-\bar G^{k-1} , \delta_k \rangle] =0.
\end{align*}
Therefore, taking expectation from both sides of \eqref{rec1} and noting \eqref{def_ek}, we obtain
\begin{align}
\sum_{k=1}^N \beta_k \alpha_k  \bbe[\|\nabla F_{\eta_k} (\theta^k)- \bar G^k\|^2] &\le  \sum_{k=1}^N \beta_k \alpha_k A_k  \left(\sum_{i=1}^k \frac{\alpha_i}{A_i} \bbe[\|e_i\|^2] + \sum_{i=1}^k \frac{\alpha_i^2}{A_i} \bbe[\|\delta_i\|^2] \right) \nn\\
&= \sum_{k=1}^N \frac{\alpha_k}{A_k}\left(\sum_{i=k}^N \beta_i \alpha_i A_i\right) \Big( \bbe[\|e_k\|^2] + \alpha_k \bbe[\|\delta_k\|^2\Big) \nn \\
&\le \sum_{k=1}^N \frac{\alpha_k \beta_k }{A_k}\left(\sum_{i=k}^N \alpha_i A_i\right) \Big( \bbe[\|e_k\|^2] + \alpha_k \bbe[\|\delta_k\|^2\Big) \nn\\
&\le c_1 \sum_{k=1}^N \alpha_k \beta_k \Big( \bbe[\|e_k\|^2] + \alpha_k \bbe[\|\delta_k\|^2\Big) \nn\\
&\le 2 c_1 \sum_{k=1}^N \left( c_2^2 \alpha_k\beta_k \bbe[\|\bar G^{k-1}\|^2] +
 \frac{4 L_0^2 d^2 \beta_k (\eta_k -\eta_{k-1})^2}{\alpha_k \eta_k^2}\right) \nn\\
 &+ 2 c_1 L_0^2(d+4)^2 \sum_{k=1}^N \beta_k \alpha_k^2,\label{f_g_differ}
\end{align}
where the second to the fourth inequalities follow from the assumptions in \eqref{choice_param}. Combining the above relation with \eqref{stationry_Geta} and \eqref{grad_upper}, we obtain \eqref{stationry_Feta}.
\end{proof}

In the next result, we specialize the rate of convergence of Algorithm~\ref{CFA_Grad_F} by specifying its parameters.
\begin{corollary}\label{cor_convg}
Let the assumptions in the statement of Theorem~\ref{theom_convg} hold and an iteration limit $N \ge 1$ is given. If the parameters are set to 
\beq \label{eta_alpha0}
\alpha_k = \frac{1}{\sqrt{\delta (d+4) N}}, \qquad 
\eta_k = \frac{\delta}{L_0 \sqrt{d}}, \qquad \beta_k = \frac{\delta}{L_0^2 d} \qquad k =1,\ldots, N,
\eeq
for some $\delta> 0$. Then we have
\begin{align}
 \bbe[\|\bar G^R\|^2] &\le
\frac{L_0^2 (d+4)^\frac32 (F^* - F(\theta^0)+5)}{\sqrt{\delta N}}, \label{stationry_Geta2}\\
 \bbe[\|\nabla F_{\eta_R} (\theta^R)\|^2] &\le
\frac{6 L_0^2 (d+4)^\frac32 (F^* - F(\theta^0)+6)}{\sqrt{\delta N}},\label{stationry_Feta2}
\end{align}
where expectation is also taken w.r.t the random integer number $R$ whose  probability distribution is supported on $\{1,\ldots,N\}$ and is given by
\beq\label{def_PR}
P_R(R=k) = \frac{\alpha_k \beta_k}{\sum_{k=1}^N \alpha_k \beta_k} \qquad k \in\{1,\ldots,N\}.
\eeq
\end{corollary}
\begin{proof}
First, note that by choices of parameters in \eqref{eta_alpha0}, we have
\begin{align*}
&\sum_{k=1}^N \beta_k \alpha_k  = \left(\frac{\delta}{L_0^2 d} \right) \left(\frac{N}{\sqrt{\delta (d+4)N}}\right) = \frac{1}{L_0^2 }\sqrt{\frac{\delta N}{d^2 (d+4)}}, \\
&\sum_{k=1}^N \beta_k \alpha_k^2  = \left(\frac{\delta}{L_0^2 d} \right) \left(\frac{N}{\delta (d+4)N}\right) \le \frac{1}{L_0^2 d (d+4)},\\
&\sum_{i=k}^N \alpha_i A_i = \sum_{i=k}^N \left(\frac{A_{i-1} -A_i}{A_{i-1}}\right) A_i
\sum_{i=k}^N \left(A_{i-1} -A_i\right) (1-\alpha_i) = \sum_{i=k}^N \left(A_i -A_{i+1}\right) \\
&= A_k -A_{N+1} \le A_k \quad \forall k \ge 1,
\end{align*}
which implies that the assumptions in \eqref{choice_param} hold with $c_1=c_2 =1$ and together with \eqref{stationry_Geta} and \eqref{stationry_Feta} imply
\eqref{stationry_Geta2} and \eqref{stationry_Feta2}.
\end{proof}

\vgap

We now add a few remarks about the abode results. First, note that \eqref{stationry_Feta} implies that the total number of function evaluations to obtain an $\epsilon$-stationary point of minimizing the smooth function ($\bar \theta \in \bbr^d$ such that $\bbe[\|\nabla F_{\eta} (\bar \theta)\|^2] \le \epsilon$)
is bounded by
\beq\label{complxty_bnd1}
{\cal O} \left(\frac{L_0^4 d^3}{\delta \epsilon^2} \right),
\eeq
which is slightly better than the one obtained in \cite{NesSpo17} (for the weighted average of $\bbe[\|\nabla F_{\eta}(\theta^k)\|^2]$ without introducing the random index $R$) in terms of dependence on $L_0$. It should be noted that due to the choice of $\eta$ in \eqref{eta_alpha0}, the parameter $\delta$ controls the error between the original objective function and its smooth approximations i.e., $|f(\theta) -f_{\frac{\delta}{L_0 \sqrt{d}}}(\theta)| \le \delta$ for any given $\theta$. Hence, as $\delta$ goes to zero, the output of Algorithm~\ref{CFA_Grad_F} will be closer to a stationary point of problem \eqref{cum_reward}.

Second, we can adaptively choose $\beta_k$ and $\eta_k$ such that they gradually converge to zero. For example, if both $\beta_k$ and $\eta_k$ are in the order of $1/k^\gamma$ for some $\gamma \in (0,1)$, the algorithm is still convergent, albeit with a worse complexity than \eqref{complxty_bnd1}. In this case, we do not need to use a very small smoothing parameter at the beginning iteration of the algorithm. 

Third, the weighted average of stochastic gradients in \eqref{grad_ave} is used to reduce the variance associated with gradient estimates. To further reduce this variance, one can use a mini-batch of samples to compute \eqnok{grad_zorth}. In particular, given a batch size of $m_k$ and generating samples $\omega^k=\{\omega^{k,i}\}_{i=1}^{m_k}$, the stochastic gradient used in \eqref{grad_ave} is computed as
\beq\label{grad_zorth_ave}
G_{\eta_k}(\theta^k,\omega^k) = \frac{1}{m_k} \sum_{i=1}^{m_k} G_{\eta_k}(\theta^k,\omega^{k,i}).
\eeq
This additional averaging will further improve the practical performance of the algorithm as shown in the next section. Also, it is worth noting that $\bbe[\|\nabla F_{\eta_R} (\theta^R)- \bar G^R \|^2]$ converges to zero with the same rate presented in Corollary~\ref{cor_convg}. Hence, $\|\bar G^k\|$ can be used as an online certificate to assess the quality of generated solutions without taking extra batch of samples. This is another advantage of using the weighted average of stochastic gradients to update the policy at each iteration of the SANG method.

Finally, when the smoothing parameter is fixed, $\beta_k$ can be set to any number while changing the rate of convergence by a constant factor. Hence, practically successful stepsize policies can be tried. For example, one can use the widely used adaptive stepsize formula in the machine learning community for stochastic optimization, namely, the Root Mean Square Propagation (RMSProp) \cite{TieHin12} given by
\beq\label{Rms}
\beta_k =  \frac{b}{\sqrt{\bar g_{k}}}, \qquad \bar g_{k}=(1-\gamma_k) \bar g_{k-1}+\gamma_k\|G^k\|^2,
\eeq
where $b$ is a tunable parameter in the learning rate $\gamma_k \in (0,1)$, and $ G^{k}$ is the gradient estimate at the $k$-th iteration. This stepsize policy performs well in our experiments as shown in the next section.

\section{Numerical Experiments}\label{sec_experiments}
In this section, we test the performance of the CFA approach in designing different parameterization policies on the energy storage problem as discussed in Section~\ref{sec_energy_storage}. To do so, we compare the objective function for a given policy $\pi$ parameterized by $\theta$ with that of the deterministic benchmark policy (defined in \eqref{DL}) which does not change the forecasts. We then report the improvement as
\beq\label{ration_improv}
\Delta F^\pi(\theta) = \frac{F^\pi (\theta) - F^{D-LA}}{|F^{D-LA}|},
\eeq
where $F^\pi$ here is the average of noisy objective values given by \eqref{cum_reward} with the contribution function \eqref{contr_func} in terms of thousand dollars. In all of our experiments, we generate $1000$ sample paths and report the averaged function evaluations to approximate the objective function i.e., $F^\pi(\theta) \approx \sum_{i=1}^{1000} \bar{F}^\pi(\theta, \omega^i)/1000$.

To better show the practical performance of the parametric CFA approach, we assume that we are only given forecasts for the renewable energy and all other information are known in the energy storage problem. In particular, we assume that the forecasts are generated according to \eqnok{wind_forecast}.

In our first set of experiments, we evaluate the performance of the constant forecast parameterization ($\pi =const$). We generate five data sets of forecasts with different levels of noise ($\rho_E \in \{0,0.1,0.2,0.3,0.4\}$) and perform a grid search for the values of $\theta$. We then compare the averaged objective values with the benchmark policy ($\theta=1$). Since the range of these values is high, we show the normalized objective improvement in Figure~\ref{fig_constant} for the purpose of better presentation. Under perfect forecast ($\rho_E=0$) the benchmark policy works the best as expected. However, under the presence of noisy forecasts, the optimal policy changes from $\theta=1$ and the constant parameterization improves the objective function.
We also examine the performance of the lookup table parameterization policy with $H = 23$ under perfect forecasts ($\rho_E=0$). In particular, we first set all values of $\theta$ to $1$ and then do a one-dimensional search over
each coordinate of $\theta$. As it can be seen from Figure~\ref{fig_constant2}, under perfect forecasts, the optimal
value for each coordinate of $\theta$ equals $1$ while the others are set to $1$.

\begin{figure}[h]
    \centering
    \includegraphics[width=4.5in,height=2.5in]{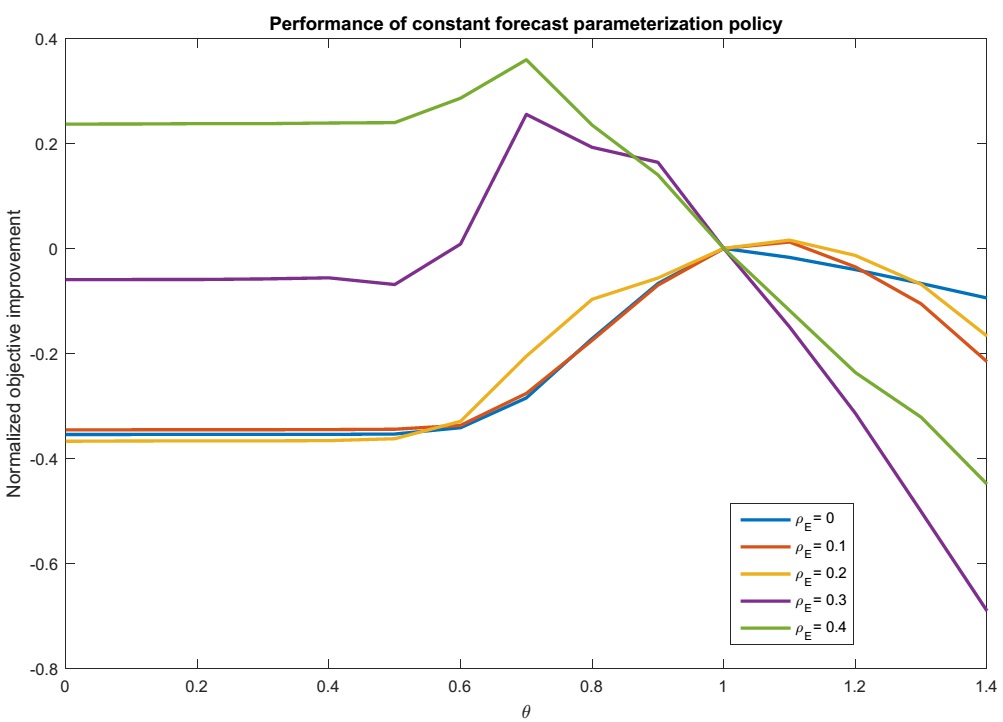}
    \caption{Performance of the constant forecast parameterization}
    \label{fig_constant}
\end{figure}

\begin{figure}[h]
        \centering
         \includegraphics[width=4.5in,height=2.5in]{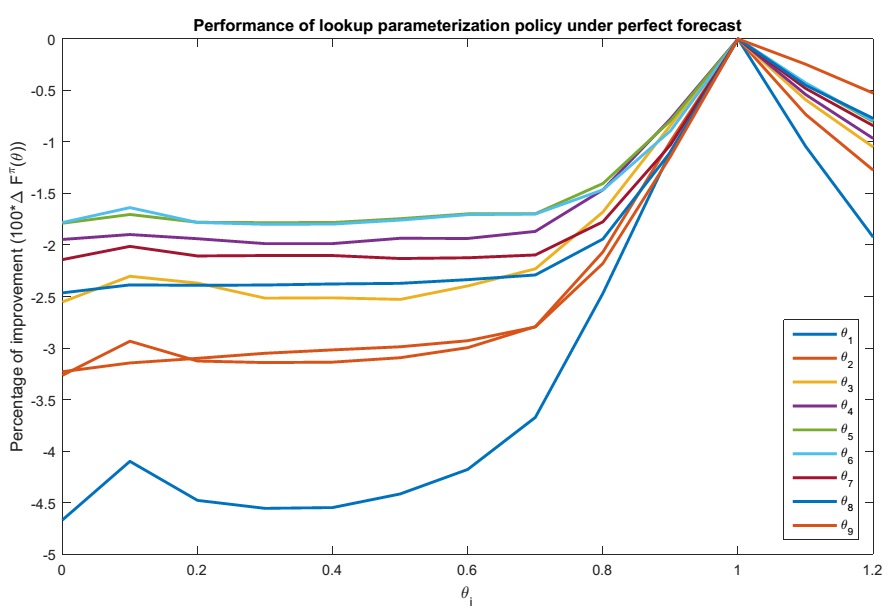}
\caption{Averaged performance of lookup parameterization policy under perfect forecasts. Each curve represents performance of the lookup policy over changing one $\theta_i$ ($i=1,\ldots,9$) while $\theta_j=1 \ \ \forall j \neq i$. The rest of $\theta_i, \ \ i>9$ has similar behaviour and are removed to increase the readability of the graph.}
         \label{fig_constant2}
\end{figure}

\begin{figure}[htp!]
    \centering
    \includegraphics[width=4.5in,height=2.5in]{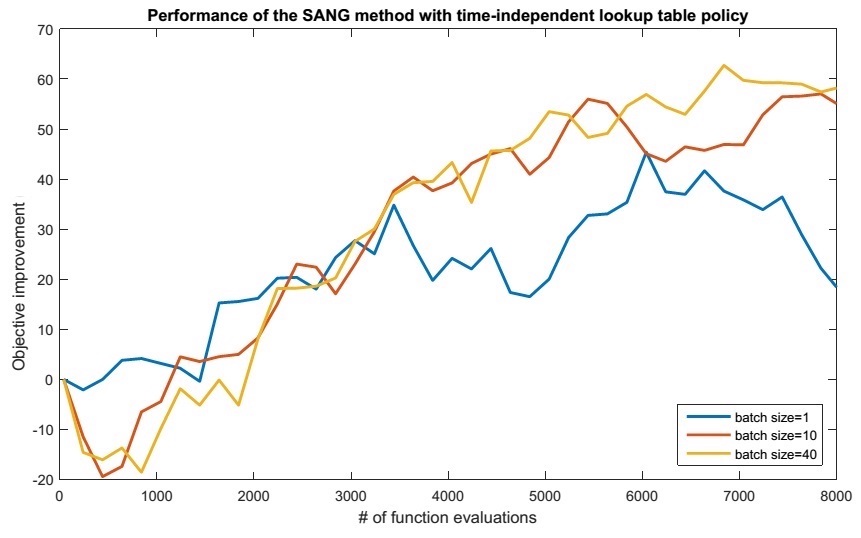}
    \includegraphics[width=4.5in,height=2.5in]{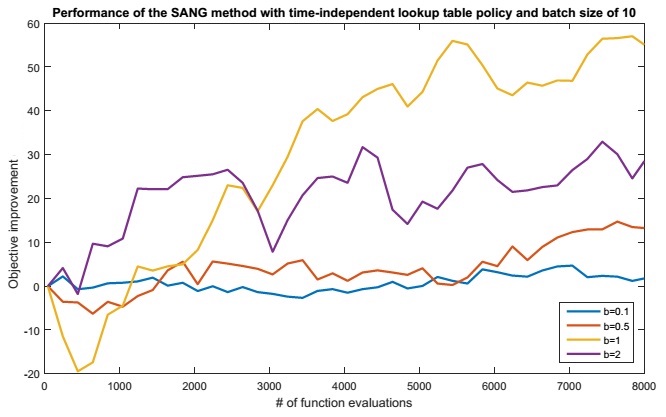}
    \includegraphics[width=4.5in,height=2.5in]{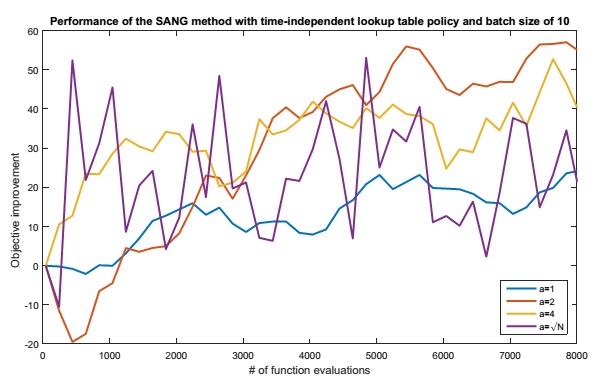}
    \caption{Performance of the SANG method with time-independent lookup table parameterization with $\rho_E=0.2$. Top: different batch sizes with $a=2$ and $b=1$; middle: different learning rates with $a=2$ and batch size of $10$; down: different tuning parameters with $b=1$ and batch size of $10$.}
    \label{fig_lookup1}
\end{figure}

In our next set of experiments, we consider the lookup table parameterization ($\pi=lkup$) which has a larger search space ($\theta \in \bbr^{23}$). Therefore, we implement Algorithm~\ref{CFA_Grad_F} with the stepsize formula in \eqref{Rms} and wight coefficients in \eqref{eta_alpha0} which multiplied by a constant factor i.e.,
\[
\alpha_k = \frac{a}{\sqrt{\delta (d+4) N}},
\]
for some $a>0$. We first run the algorithm with $a=2$, $\delta=b=1$, and different batch sizes used at each iteration to compute \eqnok{grad_zorth} using \eqnok{grad_zorth_ave}.

The top graph in Figure~\ref{fig_lookup1} shows the objective improvement over the benchmark policy vs. the number of function evaluations when using the time-independent lookup table representation of the parameters. For each choice of batch size, the algorithm runs for a different number of iterations such that the total number of function evaluations is the same  for the three runs. They all achieve at least $\$40$k improvement which is much better than the $\$5$k improvement of the constant parameterization for the same level of noise in the forecasts. Moreover, the runs with the batch sizes of $10$ and $40$ have similar and more robust convergence while the former runs for more number of iterations and hence is more likely to find better policies.

Choosing the batch size of $10$, we then run the algorithm with different learning rates and tuning parameters as shown in the middle and down graphs in Figure~\ref{fig_lookup1}. As it can be seen, the batch size of $10$ with $a=2$ and $b=1$ has the best performance among other choices.

\section{Conclusion}\label{sec_conclusion}
We provide a hybrid policy of deterministic lookahead and cost function approximations (CFA), namely, the parametric CFA to find the best policy to for energy storage problems under the presence of rolling forecasts. While this approach can handle complex stochastic models associated with the rolling forecasts, it comes at the cost of tuning parameters (policies). The objective function in the parametric CFA model is likely to be nonconvex and its unbiased gradient estimates are not easy to calculate. Hence, we present a new stochastic numerical derivative-based algorithm which only uses noisy function evaluations (obtained via simulations) to provide biased gradient estimates. By properly taking a weighted average of these biased gradient estimates, we reduce the variance associated with them which enables us to control accumulated the bias errors. Furthermore, we establish finite-time rate of convergence of this algorithm under different settings and show that it can practically find policies that performs better than the deterministic benchmark policy in optimizing an energy storage system under the presence of rolling forecasts.

\section{Acknowledgements}
The first author was partially supported by an NSERC Discovery Grant.

\bibliographystyle{apalike}
\bibliography{library}

\end{document}